\documentclass[a4paper,12pt]{article}	
\usepackage{amsmath}
\usepackage{amssymb}	
\usepackage{amsthm}\usepackage{latexsym}
\usepackage{amsmath}
\usepackage{amsthm}
\usepackage[dvipdfmx]{graphicx}
\usepackage{txfonts}
\usepackage{bm}
\usepackage{color}
\usepackage{epi	c,eepic}

\usepackage{geometry}
\numberwithin{equation}{section}

\newtheorem{theorem}{Theorem}[section]
\newtheorem{proposition}[theorem]{Proposition}

\newtheorem{lemma}[theorem]{Lemma}
\newtheorem{remark}{Remark}[section]
\newtheorem{example}{Example}[section]

\newcommand{\OMIT}[1]{{\bf [OMIT:} #1 \ {\bf --- end OMIT] }}  
   \renewcommand{\OMIT}[1]{}            

\newcommand{\RR}{{\mathbb{R}}}
\newcommand{\ZZ}{{\mathbb{Z}}}

\newcommand{\veczero}{{\bf 0}}
\newcommand{\dom}{{\rm dom\,}}

\newcommand{\unitvec}[1]{\bm{1}\sp{#1}}
\newcommand{\conv}{\Box\,}

\newcommand{\finbox}{\hspace*{\fill}$\rule{0.2cm}{0.2cm}$}

\newcommand{\Mnat}{{M$^{\natural}$}}

\newcommand{\Mnvex}{\mbox{\rm (M$\sp{\natural}$-EXC)} }

\begin{document}

\title{A Note on M-convex Functions on Jump Systems%
}

\author{
Kazuo Murota%
\thanks{Department of Economics and Business Administration,
Tokyo Metropolitan University, 
Tokyo 192-0397, Japan, 
murota@tmu.ac.jp}
}

\date{July 2019 / September 2020}

\maketitle

\begin{abstract}
A jump system is defined as a set of integer points (vectors) with a certain exchange property,
generalizing the concepts of matroids, delta-matroids,
and  base polyhedra of integral polymatroids (or submodular systems). 
A discrete convexity concept 
is defined for functions on constant-parity jump systems
and it has been used in graph theory and algebra.
In this paper we call it ``jump M-convexity'' 
and extend it to ``jump \Mnat-convexity''
for functions defined on a larger class of jump systems.
By definition, every jump M-convex function is a jump \Mnat-convex function,
and we show the equivalence of these concepts
by establishing an (injective) embedding of
jump \Mnat-convex functions in $n$ variables 
into the set of jump M-convex functions in $n+1$ variables.
Using this equivalence
we show further that jump \Mnat-convex functions admit a number of natural 
operations such as aggregation, 
projection (partial minimization), convolution, composition,
and transformation by a network.
\end{abstract}






\section{Introduction}
\label{SCintro}

A jump system \cite{BouC95} (see also  \cite{KS05jump,Lov97}) 
is defined as a set of integer points (vectors) with a certain exchange property,
generalizing the concepts of matroids, 
delta-matroids \cite{Bou87,CK88,DH86},
and  base polyhedra of integral polymatroids 
(or submodular systems) \cite{Fuj05book}. 
A concept of discrete convex functions
was introduced in \cite{Mmjump06}
for functions defined on constant-parity jump systems, 
which we call ``jump M-convex functions'' in this paper.
This is a common generalization of valuated delta-matroids \cite{DW91valdel} 
and M-convex functions \cite{Mdcasiam},
where a valuated delta-matroid is defined on an even delta-matroid and 
an M-convex function on (the integer points of) an integral base polyhedron.
We can say that a valuated delta-matroid is precisely 
the negative of a jump M-convex function 
whose domain of definition is a constant-parity jump system consisting of 0-1 vectors,
and an M-convex function is precisely 
a jump M-convex function 
whose domain of definition is a constant-sum jump system.

In this paper we extend the concept of M-convexity
to functions defined on a larger class of jump systems
that satisfy a simultaneous exchange property.
This extension corresponds, roughly, 
to the extension of M-convex functions on base polyhedra 
to \Mnat-convex functions%
\footnote{
``\Mnat''  should be pronounced as  ``em natural.''
} 
 on (generalized) polymatroids,
and accordingly, we refer to the extended concept as ``jump \Mnat-convexity.''
By the definition, described in Section \ref{SCdefs}, 
every jump M-convex function is a jump \Mnat-convex function.
The main result of this paper (Theorem \ref{THjmnatfn}) is concerned with 
the equivalence of these concepts,
establishing an (injective) embedding of
jump \Mnat-convex functions in $n$ variables 
into the set of jump M-convex functions in $n+1$ variables.

Jump M-convex functions have found applications 
in graph theory \cite{AS09convex,BK12,KST12cunconj,KT09evenf,Tak14fores}
and algebra \cite{Bra10halfplane}.
Apollonio and Seb{\H o} \cite{AS04square} (see also \cite{AS09convex}) 
considered minimization of the square-sum 
on the degree sequences of a graph,
which was the motivation of formulating the concept of
jump M-convex functions in \cite{Mmjump06}. 
A separable convex function on the degree sequences of a graph 
is a typical example of jump M-convex functions.
In a  study of matching forests, 
Takazawa \cite{Tak14fores} made use of an extension of valuated delta-matroids,
defined on delta-matroids that are not necessarily even but
are equipped with a certain simultaneous exchange property.
This motivated the present work.
A valuated delta-matroid in this extended sense is precisely 
the negative of a jump \Mnat-convex function 
whose domain of definition is contained in the unit cube (consisting of 0-1 vectors).

Our main result,
connecting jump \Mnat-convexity to jump M-convexity,
enables us to translate all results 
known for jump M-convex functions
\cite{KM07jumplink,KMT07jump,Mmjump06,MT06steep,ST07jump} 
to those for jump \Mnat-convex functions.
In this paper we are interested in 
fundamental operations such as aggregation, 
projection (partial minimization), convolution, composition,
and transformation by a network.
We investigate these operations for jump \Mnat-convex functions
on the basis of the results of  \cite{KMT07jump} 
for jump M-convex functions,
to conclude  that jump \Mnat-convex functions
admit these natural operations.
In particular, jump \Mnat-convexity is preserved
under the convolution operation (Theorem~\ref{THjmnatfnopConvol})
and the transformation by a network (Theorem~\ref{THjmnatfnopNettrans}).

This paper is organized as follows.
Section~\ref{SCdefs} gives the definitions, with the introduction of
the concept of jump \Mnat-convex functions.
Section~\ref{SCrelation} reveals precise relation between 
jump \Mnat-convexity and jump M-convexity.
Section~\ref{SCoperjmnat} deals with operations for jump \Mnat-convex functions.


\section{Definitions}
\label{SCdefs}

Let $n$ be a positive integer and $N = \{ 1,2, \ldots, n \}$.
We consider functions defined on integer lattice points, 
$f: \ZZ\sp{n} \to \RR \cup \{ +\infty \}$,
where 
the function may possibly take $+\infty$.
The {\em effective domain} of $f$ means the set of $x$
with $f(x) <  +\infty$ and is denoted 
by $\dom f =   \{ x \in \ZZ\sp{n} \mid  f(x) < +\infty \}$.
We always assume that $\dom f$ is nonempty.

For $i \in N$, the $i$th unit vector is denoted by $\unitvec{i}$.
Let $x$ and $y$ be integer vectors.
The vectors of componentwise maximum and minimum of $x$ and $y$
are denoted by $x \vee y$ and $x \wedge y$, respectively.
The smallest integer box (interval, rectangle) containing $x$ and $y$
is given by $[x \wedge y,x \vee y]_{\ZZ}$.
A vector $s \in \ZZ\sp{n}$ is called an {\em $(x,y)$-increment} if
$s= \unitvec{i}$ or $s= -\unitvec{i}$ for some $i \in N$ and
$x+s \in [x \wedge y,x \vee y]_{\ZZ}$.
An {\em $(x,y)$-increment pair} will mean
a pair of vectors $(s,t)$ such that
$s$ is an $(x,y)$-increment and $t$ is an $(x+s,y)$-increment.

A nonempty set $S \subseteq \ZZ\sp{n}$ is said 
to be a {\em jump system} 
\cite{BouC95}
if it satisfies an exchange axiom,
called the {\em 2-step axiom}:
\begin{description}
\item[(2-step axiom)]
For any  $x, y \in S$
and any $(x,y)$-increment $s$ with $x+s \not\in S$,
there exists an $(x+s,y)$-increment $t$ such that $x+s+t \in S$.
\end{description}
Note that we have the possibility of $s = t$ in the 2-step axiom.

A set $S \subseteq \ZZ\sp{n}$ is called a {\em constant-sum system}
if $x(N)=y(N)$ for any $x,y \in S$.
A constant-sum jump system is nothing but an M-convex set,
since an M-convex set is a constant-sum system
and, for a constant-sum system, the 2-step axiom 
is equivalent to the pair of following exchange axioms:
\begin{description}
\item[(B-EXC$_{+}$)] 
For any $x, y \in S$ and 
$i \in N$ with $x_{i} < y_{i}$,
there exists some 
$j \in N$ with $x_{j} > y_{j}$ such that
$x+\unitvec{i}-\unitvec{j} \in S$,

\item[(B-EXC$_{-}$)] 
For any $x, y \in S$ and 
$i \in N$ with $x_{i} > y_{i}$,
there exists some 
$j \in N$ with $x_{j} < y_{j}$ such that
$x-\unitvec{i}+\unitvec{j} \in S$,
\end{description}
each of which is known to characterize an M-convex set \cite[Theorem 4.3]{Mdcasiam}.

A set $S \subseteq \ZZ\sp{n}$ is called a {\em constant-parity system}
if $x(N)-y(N)$ is even for any $x,y \in S$.
For a constant-parity system $S$,
the 2-step axiom is simplified to:
\begin{description}
\item[(J-EXC$_{+}$)]
For any  $x, y \in S$
and for any $(x,y)$-increment $s$,
there exists an $(x+s,y)$-increment $t$ such that $x+s+t \in S$.
\end{description}
It is known (\cite[Lemma~2.1]{Mmjump06})
that this exchange property (J-EXC$_{+}$) is equivalent to 
the following (seemingly stronger) exchange property:
\begin{description}
\item[(J-EXC)]
For any  $x, y \in S$ and any $(x,y)$-increment $s$,
there exists an $(x+s,y)$-increment $t$ such that
$x+s+t \in S$ and $y-s-t \in S$.
\end{description}
That is, (J-EXC) characterizes a {\em constant-parity jump system}
(or {\em c.p.~jump system} for short).

\begin{example}  \rm  \label{EXjumpdim1}
Let  $S = \{ 0, 2 \}$, which is a subset of $\ZZ$ (with $n=1$).
This is a constant-parity jump system.
Indeed, for $(x,y,s)=(0,2,1)$ we can take $t=1$ in (J-EXC),
and for $(x,y,s)=(2,0,-1)$ we can take $t=-1$ in (J-EXC).
In contrast,
this set is not an \Mnat-convex set,
since we cannot take $t=s$ in 
the exchange axiom for an \Mnat-convex set \cite[Section 4.7]{Mdcasiam}.
\finbox
\end{example}

A function  $f: \ZZ\sp{n} \to \RR \cup \{ +\infty \}$ is called%
\footnote{
This concept (called ``jump M-convex function" here)  
was introduced in \cite{Mmjump06} under the name of
``M-convex function on a jump system."
} 
{\em jump M-convex} 
if it satisfies
the following exchange axiom:
\begin{description}
\item[(JM-EXC)]
For any  $x, y \in \dom f$
and any $(x,y)$-increment $s$,
 there exists an $(x+s,y)$-increment $t$ such that
$x+s+t \in \dom f$,  \  $y-s-t \in \dom f$, and
\begin{equation}  \label{jumpmexc2}
 f(x)+f(y) \geq  f(x+s+t)  + f(y-s-t) .
\end{equation} 
\end{description}
The effective domain of a jump M-convex function is 
a constant-parity jump system.

Just as we consider \Mnat-convex functions in addition to M-convex functions,
we can introduce the concept of {\em jump \Mnat-convex functions}
by the following exchange axiom:
\begin{description}
\item[(J\Mnat-EXC)]
For any  $x, y \in \dom f$
and any $(x,y)$-increment $s$, we have
(i) or (ii) below:
\\
(i) $x+s \in \dom f$, \  $y-s \in \dom f$, and 
\begin{equation}  \label{jumpmexc1}
 f(x)+f(y) \geq  f(x+s)  + f(y-s) , 
\end{equation}
(ii) there exists an $(x+s,y)$-increment $t$ such that
$x+s+t \in \dom f$,  \  $y-s-t \in \dom f$, and
\eqref{jumpmexc2} holds.
\end{description}
This condition 
(J\Mnat-EXC) is weaker than (JM-EXC),
and hence every jump M-convex function is a jump \Mnat-convex function.

As a consequence of (J\Mnat-EXC),
the effective domain of a jump \Mnat-convex function
is a jump system that satisfies 
\begin{description}
\item[(J$\sp{\natural}$-EXC)]
For any  $x, y \in S$ and any $(x,y)$-increment $s$, we have
(i) $x+s \in S$ and $y-s \in S$, or 
(ii)
there exists an $(x+s,y)$-increment $t$ such that
$x+s+t \in S$ and $y-s-t \in S$.
\end{description}
A jump system that satisfies
(J$\sp{\natural}$-EXC) 
will be called a {\em simultaneous exchange jump system}
(or {\em s.e.~jump system} for short).
Every constant-parity jump system
is a simultaneous exchange jump system,
since the condition (J-EXC) implies (J$\sp{\natural}$-EXC).

When a set $S$ is a subset of $\{ 0, 1 \}\sp{N}$,
$S$ is a c.p.~jump system
if and only if
it is an even delta-matroid,
and $S$ is an s.e.~jump system
if and only if
it is a simultaneous delta-matroid considered in \cite{Tak14fores},
where a subset of $N$ is identified with its characteristic vector.
Furthermore, 
when the effective domain of a function $f$ is contained in the unit cube $\{ 0, 1 \}\sp{N}$,
$f$ is jump M-convex 
if and only if
$-f$ is a valuated delta-matroid of \cite{DW91valdel},
and 
$f$ is jump \Mnat-convex 
if and only if
$-f$ is a valuation on a simultaneous delta-matroid 
in the sense of \cite{Tak14fores}.

Not every jump system is a simultaneous exchange jump system,
as the following examples show.

\begin{example}  \rm  \label{EXnonsejumpdim1}
Let  $S = \{ 0, 2, 3 \}$, which is a subset of $\ZZ$ (with $n=1$).
This set satisfies the 2-step axiom, and hence is a jump system.
However, it does not satisfy
the simultaneous exchange property (J$\sp{\natural}$-EXC).
Indeed, (J$\sp{\natural}$-EXC) fails for $x=0$, $y=3$, and $s=1$.
\finbox
\end{example}

\begin{example}[{\cite{Tak14fores}}]  \rm  \label{EXnonsejump}
Let 
$S = \{ (0,0,0), (1,1,0), (1,0,1), (0,1,1), (1,1,1) \}$.
This set satisfies the 2-step axiom, and hence is a jump system.
However, it does not satisfy
the simultaneous exchange property (J$\sp{\natural}$-EXC).
Indeed, (J$\sp{\natural}$-EXC) fails 
for $x=(0,0,0)$, $y=(1,1,1)$, and $s=(1,0,0)$.
It is worth noting that 
$S$ consists of the characteristic vectors of 
the rows (and columns) of nonsingular principal minors of 
the symmetric matrix
\[ 
A = \left[ \begin{array}{ccc}
0 & 1 & 1 \\
1 & 0 & 1 \\
1 & 1 & 0 \\
\end{array} \right]
\]
and hence it is a delta-matroid.
\finbox
\end{example}

The following example demonstrates the difference of
\Mnvex and (J\Mnat-EXC) for functions.

\begin{example}  \rm  \label{EXjmnatmnat}
Let 
$S=\{ (0,0), (1,0), (0,1), (1,1) \}$ 
and define $f: \ZZ\sp{2} \to \RR \cup \{ +\infty \}$ 
by
$f(0,0)=f(1,1)=a$ and $f(1,0)=f(0,1)=b$
with  $\dom f = S$.
\Mnvex is satisfied
if and only if $a \geq b$, whereas
(J\Mnat-EXC) is true for any $(a, b)$.
\finbox
\end{example}

The inclusion relations for sets and functions may be summarized as follows:
\begin{align*}
& 
\{ \mbox{\rm M-convex sets} \}  \subsetneqq \
  \left\{  \begin{array}{l}  \{ \mbox{\rm \Mnat-convex  sets} \}
                          \\ \{ \mbox{\rm c.p. jump systems} \} \end{array}  \right \}
\subsetneqq \   \{ \mbox{\rm s.e. jump systems} \} 
\subsetneqq \   \{ \mbox{\rm jump systems} \} ,
\\
&
\{ \mbox{\rm M-convex fns} \}  \subsetneqq \
  \left\{  \begin{array}{l}  \{ \mbox{\rm \Mnat-convex fns} \}
                          \\ \{ \mbox{\rm jump M-convex fns} \} \end{array}  \right \}
\subsetneqq \  
\{ \mbox{\rm jump \Mnat-convex fns} \} .
\end{align*}
It is noted that no convexity class is introduced for functions defined on general jump systems.


\section{Relation between jump M- and \Mnat-convex functions}
\label{SCrelation}

\subsection{Theorems}
\label{SCthmjmjmnat}

By the definitions,
jump M-convex functions are a special case of jump \Mnat-convex functions.
In this section we show that they are in fact 
equivalent to each other in the sense that 
jump \Mnat-convex function in $n$ variables 
can be identified with jump M-convex functions in $n+1$ variables.

For any integer vector $x \in \ZZ\sp{n}$
we introduce a notation  $\pi(x)$ to indicate
the parity of its component sum $x(N)$.
That is, we define $\pi(x)=0$ if $x(N)$ is even, and
$\pi(x)=1$ if $x(N)$ is odd.
For any set $S \subseteq \ZZ\sp{n}$ we associate a set
$\tilde S \subseteq \ZZ\sp{n+1}$ defined by
\begin{equation}  \label{Stilde}
 \tilde S = \{ (x_{0}, x) \in \ZZ \times \ZZ\sp{n} \mid
  x_{0} = \pi(x), x \in S \} ,
\end{equation}  
which is a constant-parity system.
It is pointed out by J. Geelen \cite{Gee96prcom} that
(J$\sp{\natural}$-EXC) for $S$ is equivalent to (J-EXC) for $\tilde S$.

\begin{theorem}[{\cite{Gee96prcom}}] \label{THjnatset}
$S$ satisfies {\rm (J$\sp{\natural}$-EXC)} if and only if 
$\tilde S$ satisfies {\rm (J-EXC)}.
That is,
$S$ is a simultaneous exchange jump system 
if and only if
$\tilde S$ is a constant-parity jump system.
\end{theorem}
\begin{proof}
The proof is given in  Section~\ref{SCproofjnatset}.
\end{proof}

For any function 
$f: \ZZ\sp{n} \to \RR \cup \{ +\infty \}$,
we associate a function 
$\tilde f: \ZZ\sp{n+1} \to \RR \cup \{ +\infty \}$
defined by
\begin{equation}  \label{ftilde}
\tilde f(x_{0}, x)= 
   \left\{  \begin{array}{ll}
   f(x)            &   (x_{0} = \pi(x)) ,      \\
   + \infty      &   (\mbox{\rm otherwise}) , \\
                      \end{array}  \right.
\end{equation}
where $x_{0} \in \ZZ$ and $x \in \ZZ\sp{n}$.
Note that $\dom \tilde f$ is a constant-parity system,
and it coincides with $\tilde S$ associated with $S = \dom f$ 
as in \eqref{Stilde}.

The following theorem is a main result of this paper, showing that 
$f$ is jump \Mnat-convex 
if and only if
$\tilde f$ is jump M-convex.

\begin{theorem} \label{THjmnatfn}
$f$ satisfies {\rm (J\Mnat-EXC)} if and only if 
$\tilde f$ satisfies {\rm (JM-EXC)}.
That is,
$f$ is jump \Mnat-convex 
if and only if
$\tilde f$ is jump M-convex.
\end{theorem}
\begin{proof}
The proof is given in  Section~\ref{SCproofjmnatfn}.
\end{proof}

This theorem enables us to translate all results 
known for jump M-convex functions
\cite{KM07jumplink,KMT07jump,Mmjump06,MT06steep,ST07jump} 
to those for jump \Mnat-convex functions.
In Section~\ref{SCoperjmnat} we will investigate 
fundamental operations for jump \Mnat-convex functions
on the basis of the results of  \cite{KMT07jump} 
for jump M-convex functions.

\subsection{Proof of of Theorem~\ref{THjnatset}}
\label{SCproofjnatset}

The proof of Theorem~\ref{THjnatset} is provided here.
Recall from (\ref{Stilde}) that
\begin{equation}  \label{Stilde2}
 \tilde S = \{ (x_{0}, x) \in \ZZ \times \ZZ\sp{n} \mid
  x_{0} = \pi(x), x \in S \},
\end{equation}  
where $\pi(x)=0$ if $x(N)$ is even, and
$\pi(x)=1$ if $x(N)$ is odd.
The set $\tilde S$ determines $S$ uniquely.

The ``if'' part is stated in Lemma \ref{LMjmjmnatset} below,
whereas the ``only if'' part follows from Lemma \ref{LMjnatmjmset} 
together with the equivalence 
of (J-EXC) and (J-EXC$_{+}$)
given in \cite[Lemma~2.1]{Mmjump06}.

\begin{lemma}   \label{LMjmjmnatset}
If $\tilde S$ satisfies {\rm (J-EXC)},
then $S$ satisfies {\rm (J$\sp{\natural}$-EXC)}.
\end{lemma}
\begin{proof}
Let $x, y \in S$ and $s$ be an $(x,y)$-increment.
Let $\tilde x := (\pi(x), x)$, $\tilde y := (\pi(y), y)$, and $\tilde s := (0,s)$,
where $\tilde s$ is an $(\tilde x, \tilde y)$-increment.
By (J-EXC) for $\tilde S$
there exists an $(\tilde x + \tilde s,\tilde y)$-increment $\tilde t=(t_{0},t)$
such that
$\tilde x + \tilde s + \tilde t \in \tilde S$ and
$\tilde y -\tilde s - \tilde t \in \tilde S$.
If $t=0$, then $x+s \in S$ and $y-s \in S$.
If $t\not=0$, then $t$ is an $(x+s,y)$-increment, 
and $x+s+t \in S$ and $y-s-t \in S$.
Thus $S$ satisfies {\rm (J$\sp{\natural}$-EXC)}.
\end{proof}

\begin{lemma}   \label{LMjnatmjmset}
If $S$ satisfies {\rm (J$\sp{\natural}$-EXC)},
then $\tilde S$ satisfies {\rm (J-EXC$_{+}$)}.
\end{lemma}
\begin{proof}
(This proof is due to J. Geelen \cite{Gee96prcom}.)
\\
\noindent
{Claim A:} For $x,y \in S$ with $\pi(x) \not= \pi(y)$,
there exists an $(x,y)$-increment $s$ such that $x+s \in S$.

This claim can be proved by induction on $\| x-y\| _{1}$.
Let $s$ be an $(x,y)$-increment. If $x+s \in S$, we are done.
Otherwise, by (J$\sp{\natural}$-EXC), there exists 
an $(x+s,y)$-increment $t$ with $y-s-t \in S$.
Since $\pi(x) \not= \pi(y-s-t)$ and $\| x-(y-s-t)\| _{1} <  \| x-y\| _{1} $,
there exists, by induction, 
an $(x,y-s-t)$-increment $s'$ such that $x+s' \in S$,
where $s'$ is obviously an $(x,y)$-increment.
Thus the claim is established.

Let $\tilde x, \tilde y \in \tilde S$ and
$\tilde s=(s_{0}, s)$ be an $(\tilde x, \tilde y)$-increment.
We have $\tilde x =(\pi(x),x)$ and $\tilde y =(\pi(y),y)$ 
with $x,y \in S$.
If $s=0$, then $\pi(x) \not= \pi(y)$ and the claim shows
the existence of an $(x,y)$-increment $t$ with $x+t \in S$.
Let $z:=x+t$, $\tilde z :=(\pi(z),z)$ and $\tilde t :=(0,t)$. 
Then $\tilde t$ is an $(\tilde x + \tilde s, \tilde y)$-increment 
and  
$\tilde z = \tilde x + \tilde s + \tilde t$.
Moreover we have $\tilde z \in \tilde S$,
since $z = x+t \in S$.
In what follows we assume $s\not=0$ and hence $s_{0}=0$.

Suppose that $x+s \in S$.
If $\pi(x)\not=\pi(y)$, we can take
an $(\tilde x + \tilde s, \tilde y)$-increment 
$\tilde t = (t_{0},0)$ with $t_{0} \not=0$,
for which
$\tilde x + \tilde s + \tilde t  = (\pi(x) + t_{0},x+s) 
 = (\pi(x+s),x+s) \in \tilde S$.
Otherwise ($\pi(x)=\pi(y)$), we apply Claim A to $(x+s,y)$
to obtain an $(x+s,y)$-increment $t$ with $x+s+t \in S$.
Then $\tilde t :=(0,t)$ is an $(\tilde x + \tilde s, \tilde y)$-increment
and 
$\tilde x + \tilde s + \tilde t 
= (\pi(x), x+s+t) = (\pi(x+s+t), x+s+t)\in \tilde S$.

Finally suppose that $x+s \not \in S$.
It follows from (J$\sp{\natural}$-EXC) that 
there exists an $(x+s,y)$-increment $t$ with $x+s+t \in S$.
Then $\tilde t =(0,t)$ is an $(\tilde x + \tilde s, \tilde y)$-increment
and $\tilde x + \tilde s + \tilde t \in \tilde S$.
\end{proof}

\subsection{Proof of Theorem~\ref{THjmnatfn}}
\label{SCproofjmnatfn}	

The proof of Theorem~\ref{THjmnatfn} is provided here.
Recall from (\ref{ftilde}) that
\begin{equation}  \label{ftilde2}
\tilde f(x_{0}, x)= 
   \left\{  \begin{array}{ll}
   f(x)            &   (x_{0} = \pi(x)) ,      \\
   + \infty      &   (\mbox{\rm otherwise}) . \\
                      \end{array}  \right.
\end{equation}
Let $S = \dom f$ and $\tilde S = \dom \tilde f$.

The ``if'' part is established in Lemma \ref{LMjmjmnatfn} below.

\begin{lemma}   \label{LMjmjmnatfn}
If $\tilde f$ satisfies {\rm (JM-EXC)},
then $f$ satisfies {\rm (J\Mnat-EXC)}.
\end{lemma}
\begin{proof}
Let $x, y \in S$ and $s$ be an $(x,y)$-increment.
Let $\tilde x := (\pi(x), x)$, $\tilde y := (\pi(y), y)$, and $\tilde s := (0,s)$,
where $\tilde s$ is an $(\tilde x, \tilde y)$-increment.
By (JM-EXC) for $\tilde f$
there exists an $(\tilde x + \tilde s,\tilde y)$-increment $\tilde t=(t_{0},t)$
such that
$\tilde x + \tilde s + \tilde t \in \tilde S$,
$\tilde y -\tilde s - \tilde t \in \tilde S$, and
\[
  \tilde f( \tilde x ) + \tilde f( \tilde y )
 \geq
  \tilde f( \tilde x +\tilde s + \tilde t \, )
+ \tilde f( \tilde y -\tilde s - \tilde t \, ).
\]
If $t=0$, then $x+s \in S$, $y-s \in S$, and \eqref{jumpmexc1} holds,
which is the case (i) in (J\Mnat-EXC).
If $t\not=0$, then $t$ is an $(x+s,y)$-increment, 
$x+s+t \in S$, $y-s-t \in S$, and \eqref{jumpmexc2} holds,
which is the case (ii) in (J\Mnat-EXC).
Thus $f$ satisfies {\rm (J\Mnat-EXC)}.
\end{proof}

To prove the ``only if'' part of Theorem~\ref{THjmnatfn},
suppose that $f$ satisfies  (J\Mnat-EXC).
Then $S =\dom f$
 satisfies (J$\sp{\natural}$-EXC),
and hence
$\tilde S$ is a constant-parity jump system by Theorem~\ref{THjnatset}.
The ``only if'' part is established by Lemma \ref{LMjnatmjmfn} below
on the basis of the following local characterization of jump M-convexity.

\begin{lemma}[\protect{\cite[Theorem 2.3]{Mmjump06}}] \label{LMjmlocex}
A function 
$f: \ZZ\sp{n} \to \RR \cup \{ +\infty \}$
is jump M-convex
if and only if $\dom f$ is 
a constant-parity jump system and
$f$ satisfies the local exchange property: 
\begin{description}
\item[(JM-EXC$_{\rm loc}$)] 
For any  $x, y \in \dom f$ with $\|x - y \|_{1} = 4$
there exists an $(x,y)$-increment pair  $(s,t)$ such that
$x+s+t \in \dom f$, \  $y-s-t \in \dom f$, and
\eqref{jumpmexc2} holds.
\end{description}
\vspace{-\baselineskip} 
\finbox
\end{lemma}

\begin{lemma}   \label{LMjnatmjmfn}
If $f$ satisfies {\rm (J\Mnat-EXC)},
then $\tilde f$ satisfies {\rm (JM-EXC$_{\rm loc}$)}.
\end{lemma}
\begin{proof}
Take $\tilde x=(\pi(x),x)$ and $\tilde y=(\pi(y),y) \in \tilde S$ 
with $\| \tilde x - \tilde y\| _{1} = 4$.
We should find an $(\tilde x,\tilde y)$-increment pair 
$(\tilde s, \tilde t \,)$ 
such that
\begin{equation}  \label{tildefjumpmexc2}
\tilde f(\tilde x) + \tilde f(\tilde y)  
\geq  
\tilde f(\tilde x +\tilde s+\tilde t \,)  
+ \tilde f(\tilde y-\tilde s-\tilde t \,) .
\end{equation} 
Note that $\tilde f(\tilde x) + \tilde f(\tilde y)$
is equal to $f(x)+f(y)$.

If $\pi(x) \not= \pi(y)$, we have $\| x - y\| _{1} = 3$
and 
$y= x+ s_{1} + s_{2} + s_{3}$
with $s_{i} \in \ZZ\sp{n}$ and  $\| s_{i}\| _{1}=1$ for  $i=1,2,3$,
where the vectors $s_{1},s_{2},s_{3}$ are not necessarily distinct.
By (J\Mnat-EXC) for $f$, there exists an $(x,y)$-increment $s$ such that
\begin{equation} \label{prfLem37}
 f(x)+f(y) \geq  f(x+s)  + f(y-s) .
\end{equation}
Indeed, if this inequality is not true for $s=s_{1}$, then (J\Mnat-EXC) implies 
$f(x)+f(y) \geq  f(x+s_{1} +t)  + f(y-s_{1} -t)$
for some $t \in \{ s_{2}, s_{3} \}$.
Let $s = s_{3}$ if $t = s_{2}$, and
$s = s_{2}$ if $t = s_{3}$.
Then $f(x+s_{1}+t) = f(y - s)$ and  $f(y-s_{1}-t) = f(x + s)$,
and the inequality \eqref{prfLem37} holds.
For $s$ satisfying \eqref{prfLem37},
$(\tilde s, \tilde t \,) =((0,s), (t_{0},0))$ 
with $t_{0} \in \{ +1 , -1 \}$
is a desired increment pair, since
$f(x+s)= \tilde f(\tilde x +\tilde s+\tilde t \,)$
and
$f(y-s) = \tilde f(\tilde y-\tilde s-\tilde t \,)$.

If $\pi(x) = \pi(y)$, we have $\| x - y\| _{1} = 4$
and 
$y= x+ s_{1} + s_{2} + s_{3} + s_{4}$
with $s_{i} \in \ZZ\sp{n}$ and  $\| s_{i}\| _{1}=1$ for  $i=1,2,3,4$,
where the vectors $s_{1},s_{2},s_{3},s_{4}$
are not necessarily distinct.
Using a short-hand notation 
 $f_{\alpha_{1} \alpha_{2} \alpha_{3} \alpha_{4}}$
for $f(x + \sum_{i=1}\sp{4} \alpha_{i} s_{i})$,
where $\alpha_{i} \in \{ 0,1 \}$,
we want to show 
\begin{equation} \label{mnatproof3}
f_{0 0 0 0}+f_{1 1 1 1}  \geq
  \min( f_{1 1 0 0}+f_{0 0 1 1},  \  f_{1 0 1 0}+f_{0 1 0 1},  \ 
  f_{1 0 0 1}+f_{0 1 1 0}  ) .
\end{equation}
This implies that 
$(\tilde s, \tilde t \,) =((0,s_{i}), (0,s_{j}))$
is a desired increment pair for some distinct $i$, $j$.

To prove (\ref{mnatproof3}) by contradiction, assume that
\begin{equation} \label{mnatproof3neg}
f_{0 0 0 0}+f_{1 1 1 1}  < 
  \min( f_{1 1 0 0}+f_{0 0 1 1}, \   f_{1 0 1 0}+f_{0 1 0 1},  \
  f_{1 0 0 1}+f_{0 1 1 0}  ) .
\end{equation}
It is convenient here
to imagine an undirected graph $G$ with vertex set $V = \{ 1,2,3,4 \}$
and edge set $E = \{ (i,j) \mid x+s_{i}+s_{j} \in \dom f \}$.
To edge $(i,j)$ we assign weight $w_{ij}=f(x+s_{i}+s_{j})$.
Then the right-hand side of (\ref{mnatproof3neg}) is equal to
the minimum weight of a perfect matching (of size two).
If no perfect matching exists in $G$, 
the right-hand side of (\ref{mnatproof3neg}) is equal to $+\infty$ (by convention).

Suppose first that $G$ admits a perfect matching. 
By duality there exist real numbers (``potentials'') $p_{1},p_{2},p_{3},p_{4}$ 
associated with the vertices such that
\begin{equation} \label{prfwijpipj}
 f(x+s_{i}+s_{j}) \geq p_{i} + p_{j}
 \qquad ( i \not= j)
\end{equation}
and
\[
p_{1} + p_{2} +p_{3} + p_{4}
=  \min( f_{1 1 0 0}+f_{0 0 1 1}, \  f_{1 0 1 0}+f_{0 1 0 1},  \  f_{1 0 0 1}+f_{0 1 1 0}  ) 
\]
(see Remark \ref{RMblossom}).
Then it follows from (\ref{mnatproof3neg}) that
\begin{equation} \label{prff0f1p1p4}
f_{0 0 0 0}+f_{1 1 1 1}  <  p_{1} + p_{2} +p_{3} + p_{4}.
\end{equation}
Even when $G$ has no perfect matching, it is easy to verify that 
we can take $p_{1},p_{2},p_{3},p_{4}$ satisfying
\eqref{prfwijpipj} and \eqref{prff0f1p1p4}.

We introduce notation 
\[
g_{\alpha_{1} \alpha_{2} \alpha_{3} \alpha_{4}} =
f_{\alpha_{1} \alpha_{2} \alpha_{3} \alpha_{4}} - \sum_{i=1}\sp{4} \alpha_{i} p_{i},
\]
where $\alpha_{i}  \in \{ 0,1 \}$ $(i=1,2,3,4)$.
For example.
\begin{align} 
&
g_{0 0 0 0} = f_{0 0 0 0},
\quad  
g_{1 1 1 1} = f_{1 1 1 1} - p_{1} - p_{2} - p_{3} - p_{4} , 
\label{prfg0000g1111}
\\ &
 g_{1 0 0 0} = f_{1 0 0 0} - p_{1},
\quad  
g_{1 1 0 0} = f_{1 1 0 0} - p_{1} - p_{2}, 
\quad  
g_{0 1 1 1} = f_{0 1 1 1} - p_{2} - p_{3} - p_{4} .
\end{align}
With this notation we can rewrite \eqref{prfwijpipj} as
\begin{equation} \label{prfgij00pos}
g_{1 1 0 0} \geq 0,
\quad  
g_{1 0 1 0} \geq 0,
\quad  
g_{1 0 0 1} \geq 0,
\quad  
g_{0 1 1 0} \geq 0,
\quad  
g_{0 1 0 1} \geq 0,
\quad  
g_{0 0 1 1} \geq 0.
\end{equation}
In addition, we assume without loss of generality that
\begin{equation} \label{prfg1000min}
  \min(   g_{1 0 0 0}, \ g_{0 1 0 0},  \   g_{0 0 1 0},  \  g_{0 0 0 1}  )
 = g_{1 0 0 0}.
\end{equation}

By the exchange property (J\Mnat-EXC) we have
\[
f_{0 0 0 0}+f_{1 1 1 0}  \geq
  \min( 
          f_{1 0 0 0}+f_{0 1 1 0}, \
          f_{0 1 0 0}+f_{1 0 1 0}, \
          f_{0 0 1 0}+f_{1 1 0 0}  
  ) .
\]
By subtracting $p_{1} +p_{2} + p_{3}$
from both sides and 
 using \eqref{prfgij00pos} and \eqref{prfg1000min},
we obtain
\begin{align*}
g_{0 0 0 0}+g_{1 1 1 0} 
 &\geq
  \min(   g_{1 0 0 0}+g_{0 1 1 0}, \
          g_{0 1 0 0}+g_{1 0 1 0}, \
          g_{0 0 1 0}+g_{1 1 0 0}   )
\\
 &\geq
  \min(     g_{1 0 0 0}, \   g_{0 1 0 0}, \ g_{0 0 1 0}    )
 = g_{1 0 0 0},
\end{align*}
from which 
\[
g_{1 1 1 0}  \geq g_{1 0 0 0} - g_{0 0 0 0}.
\]
Similarly, we obtain
$g_{1 0 1 1}  \geq g_{1 0 0 0} - g_{0 0 0 0}$ and 
$g_{1 1 0 1}  \geq g_{1 0 0 0} - g_{0 0 0 0}$,
and therefore
\begin{equation} \label{prfg1000g0000}
  \min(   g_{1 0 1 1}, \  g_{1 1 0 1},  \    g_{1 1 1 0}  )
 \geq g_{1 0 0 0} - g_{0 0 0 0}.
\end{equation}

Again by (J\Mnat-EXC) we have
\[
f_{1 1 1 1}+f_{1 0 0 0}  \geq
  \min(   f_{1 1 0 0}+f_{1 0 1 1}, \
          f_{1 0 1 0}+f_{1 1 0 1}, \ 
          f_{1 0 0 1}+f_{1 1 1 0}  ).
\] 
By subtracting $2 p_{1} + p_{2} +p_{3} + p_{4}$ from both sides and 
using \eqref{prfgij00pos} and \eqref{prfg1000g0000} 
we obtain
\begin{align*}
g_{1 1 1 1}+g_{1 0 0 0}  
 &\geq
  \min(   g_{1 1 0 0}+g_{1 0 1 1}, \
          g_{1 0 1 0}+g_{1 1 0 1}, \ 
          g_{1 0 0 1}+g_{1 1 1 0}  )
\\
 &\geq
  \min(   g_{1 0 1 1}, \   g_{1 1 0 1},  \    g_{1 1 1 0}  )
\\
 & \geq g_{1 0 0 0} - g_{0 0 0 0} ,
\end{align*}
from which 
\[
 g_{0 0 0 0}  + g_{1 1 1 1}  \geq 0  .
\]
By \eqref{prfg0000g1111}
this is equivalent to
\[
 f_{0 0 0 0} + f_{1 1 1 1}  \geq p_{1} + p_{2} +p_{3} + p_{4}  ,
\]
which is a contradiction to \eqref{prff0f1p1p4}.
\end{proof}

\begin{remark}  \rm  \label{RMblossom}
According to Edmonds' perfect matching polytope theorem
\cite[Theorem 25.1]{Sch03}, 
the perfect matchings on $G=(V,E)$
can be described by a system of inequalities consisting of three kinds of inequalities:
$x_{e} \geq 0$ for $e \in E$,
$x(\delta (v)) =1$ for $v \in V$, and 
$x(\delta (U)) \geq 1$
for $U \subseteq V$ with odd $|U| \geq 3$,
where $x \in \RR\sp{E}$ and 
$\delta (U)$ denotes the set of edges connecting $U$ and $V \setminus U$.
In our present case, we have $|V|=4$, and hence we do not need the inequality 
$x(\delta (U)) \geq 1$ for $U$ with $|U| = 3$,
since $x(\delta (U)) = x(\delta (v))$ for $v \in V \setminus U$.
\finbox
\end{remark}

\section{Operations for jump \Mnat-convex functions}
\label{SCoperjmnat}
	
In this section we consider fundamental operations
such as addition, projection, aggregation, and convolution
for jump \Mnat-convex functions.
Our result (Theorem~\ref{THjmnatfn}),
connecting jump \Mnat-convexity to jump M-convexity,
enables us to translate 
the results of \cite{KMT07jump}
for jump M-convex functions
to those for jump \Mnat-convex functions.
The main objective is to show that 
jump \Mnat-convexity is preserved
under the convolution operation (Theorem~\ref{THjmnatfnopConvol})
and the transformation by a network (Theorem~\ref{THjmnatfnopNettrans}).

\subsection{Basic operations}
\label{SCbasicope}

We start with basic operations for a jump \Mnat-convex function.

\begin{proposition}  \label{PRjmnatfnop1}
Let $f: \ZZ\sp{n} \to \RR \cup \{ +\infty \}$
be a jump \Mnat-convex function.

\noindent
{\rm (1)}
For an integer vector $b$, \ 
$g(x) = f(x-b)$ is jump \Mnat-convex.

\noindent
{\rm (2)}
For any $\tau_{i} \in \{ +1, -1 \}$ 
$(1 \leq i \leq n)$, \ 
$g(x_{1},x_{2}, \ldots, x_{n}) = f(\tau_{1} x_{1}, \tau_{2} x_{2}, \ldots,  \tau_{n}x_{n})$
is jump \Mnat-convex.
In particular, \  $g(x) = f(-x)$ is jump \Mnat-convex.

\noindent
{\rm (3)}
For any permutation $\sigma$ of $(1,2,\ldots,n)$, \ 
$g(x_{1},x_{2}, \ldots, x_{n}) =  f(x_{\sigma(1)}, x_{\sigma(2)}, \ldots, x_{\sigma(n)}) $
is jump \Mnat-convex.
\end{proposition}
\begin{proof}
(1)--(3)
We can verify easily that $g$ satisfies the exchange property (J\Mnat-EXC).
\end{proof}

A function
$\varphi: \ZZ^{n} \to \RR \cup \{ +\infty \}$
in $x=(x_{1}, x_{2}, \ldots,x_{n}) \in \ZZ^{n}$
is called  {\em separable convex}
if it can be represented as
\begin{equation}  \label{sepfndef}
\varphi(x) = \varphi_{1}(x_{1}) + \varphi_{2}(x_{2}) + \cdots + \varphi_{n}(x_{n})
\end{equation}
with univariate functions
$\varphi_{i}: \ZZ \to \RR \cup \{ +\infty \}$ satisfying 
\begin{equation}  \label{univarconvdef}
\varphi_{i}(t-1) + \varphi_{i}(t+1) \geq 2 \varphi_{i}(t)
\qquad (t \in \ZZ).
\end{equation}

\begin{proposition} \label{PRjmnatfnopAdd}
Let $f: \ZZ\sp{n} \to \RR \cup \{ +\infty \}$ be a jump \Mnat-convex function.

\noindent
{\rm (1)}
For any $a \geq 0$, \ 
$g(x) = a f(x)$ 
is jump \Mnat-convex.

\noindent
{\rm (2)}
For any $c \in \RR\sp{n}$, \ 
$g(x) = f(x) + \sum_{i=1}\sp{n} c_{i} x_{i}$
is jump \Mnat-convex.

\noindent
{\rm (3)}
For any separable convex function $\varphi$,  \ 
$g(x) = f(x) + \varphi(x)$
is jump \Mnat-convex.
\end{proposition}
\begin{proof}
For (1) and (2)
we can verify easily that $g$ satisfies the exchange property (J\Mnat-EXC).
For (3) it suffices to prove that 
the function $g(x) = f(x) + \varphi_{i}(x_{i})$ 
with a particular $i\in N$
satisfies (J\Mnat-EXC).
Suppose that $x\in \ZZ\sp{n}, \ y \in \ZZ\sp{n}$ and $s$ is an $(x, y)$-increment. 
By (J\Mnat-EXC) for $f$, 
we have (i)
$f(x) + f(y) \geq f(x+s) + f(y-s)$ 
or (ii)
$f(x) + f(y) \geq f(x+s+t) + f(y-s-t)$
 for some $(x+s, y)$-increment $t$.
In the former case (i)
we have
$\varphi_{i}(x_{i}) + \varphi_{i}(y_{i}) 
\geq  \varphi_{i}(x_{i}+s_{i}) + \varphi_{i}(y_{i}-s_{i})$
by convexity of $\varphi_{i}$, and hence
$g(x) + g(y) \geq g(x+s) + g(y-s)$.
In the latter case (ii)
we have
$\varphi_{i}(x_{i}) + \varphi_{i}(y_{i}) 
\geq  \varphi_{i}(x_{i}+s_{i}+t_{i}) + \varphi_{i}(y_{i}-s_{i}-t_{i})$
by convexity of $\varphi_{i}$, and hence
$g(x) + g(y) \geq g(x+s+t) + g(y-s-t)$.
Thus $g$ satisfies (J\Mnat-EXC).
\end{proof}

It is noted that the scaling operation of the variables 
does not preserve jump \Mnat-convexity.
That is, for a positive integer $\alpha$, the function 
$g(x) = f(\alpha x)$ in $x \in \ZZ\sp{n}$
is not necessarily jump \Mnat-convex. 
In connection to (3) above, it is noteworthy that
the sum $f_{1} + f_{2}$ of two jump \Mnat-convex functions $f_{1}$ and $f_{2}$
is not necessarily jump \Mnat-convex.

Let $U$ be a subset of $N = \{ 1,2, \ldots, n \}$, and 
$f: \ZZ\sp{N} \to \RR \cup \{ +\infty \}$.
The {\em projection} of $f$ to $U$ is a function
$f\sp{U}: \ZZ\sp{U} \to \RR \cup \{ +\infty, -\infty \}$
defined by
\begin{equation} \label{projfndef} 
  f\sp{U}(y)  =  \inf \{ f(y,z) \mid z \in \ZZ\sp{N \setminus U} \}
  \qquad (y \in \ZZ\sp{U}) ,
\end{equation}
where the notation $(y,z)$ means the vector
whose $i$th component is equal to $y_{i}$ for $i \in U$
and to $z_{i}$ for $i \in N \setminus U$;
for example, if $N = \{ 1,2,3,4 \}$ and $U = \{ 2,3 \}$,
we have $(y,z) = (z_{1}, y_{2}, y_{3}, z_{4})$.
The projection is sometimes called {\em partial minimization}.
The
{\em restriction} of $f$ to $U$ is a function 
$f_{U}: \ZZ\sp{U} \to \RR \cup \{ +\infty \}$ defined by
\begin{equation} \label{restrfndef}
 f_{U}(y) = f(y,\veczero_{N \setminus U})
  \qquad (y \in \ZZ\sp{U}) ,
\end{equation}
where $\veczero_{N \setminus U}$ denotes the zero vector 
in $\ZZ\sp{N \setminus U}$.

\begin{proposition} \label{PRjmnatfnopRestProj}
Let 
$f: \ZZ\sp{N} \to \RR \cup \{ +\infty \}$ be a jump \Mnat-convex function, and $U \subseteq N$.

\noindent
{\rm (1)}
The restriction $f_{U}$ is jump \Mnat-convex, provided that $\dom f_{U} \neq \emptyset$.

\noindent
{\rm (2)}
The projection $f\sp{U}$ is jump \Mnat-convex, provided that $f\sp{U} > -\infty$.
\end{proposition}
\begin{proof}
(1) 
We can verify easily that the restriction $f_{U}$ satisfies
the exchange property (J\Mnat-EXC).
(2) 
This follows from Theorem~\ref{THjmnatfnopConvol} 
since the projection can be regarded as a special case of convolution.
See Remark \ref{RMprojFromConvol} in Section~\ref{SCconvol}.
\end{proof}

For a set $S \subseteq \ZZ\sp{N}$
and a subset $U$ of $N$,
the {\em restriction} 
of $S$ to $U$ is a subset of $\ZZ\sp{U}$ defined by
\begin{equation} \label{restrsetdef}
 S_{U} =  \{ y \in \ZZ\sp{U} \mid  (y,\veczero_{N \setminus U}) \in S \}, 
\end{equation}
and the {\em projection}
of $S$ to $U$ is a subset of $\ZZ\sp{U}$ defined by
\begin{equation}  \label{projsetdef}
 S\sp{U} =  \{ y \in \ZZ\sp{U} \mid  (y,z) \in S \mbox{ for some $z \in \ZZ\sp{N \setminus U}$} \}. 
\end{equation}
Since a set satisfies 
the exchange property (J$\sp{\natural}$-EXC)
precisely when its indicator function  satisfies (J\Mnat-EXC),
it follows from Proposition~\ref{PRjmnatfnopRestProj}
that 
the restriction of an s.e.~jump system, if not empty,
is an s.e.~jump system,
and that 
the projection of an s.e.~jump system
is an s.e.~jump system.

\subsection{Aggregation}
\label{SCaggr}

Let $\mathcal{P}$ be a  partition of $N = \{ 1,2, \ldots, n \}$ into disjoint (nonempty) subsets:
$N = N_{1} \cup N_{2} \cup \dots \cup N_{m}$.
For a function $f: \ZZ \sp{N} \to \RR \cup \{+\infty \}$, 
the {\em aggregation} of $f$ with respect to $\mathcal{P}$
is defined as a function
$g : \ZZ \sp{m} \to \RR \cup \{+\infty, -\infty\}$ given by 
\begin{equation} \label{aggrfndef}
g(y_{1}, y_{2}, \dots , y_{m}) 
= \inf \{ f(x)   \mid x(N_{j}) = y_{j} \ (j=1,2,\ldots,m) \}. 
\end{equation}
If $m = n-1$
(in which case we have $|N_{k}|=2$ for some $k$ and $|N_{j}| =1$ for other $j\not= k$),
 this is called an {\em elementary aggregation}.
Any (general) aggregation can be obtained by repeated applications of elementary aggregations.

For jump M-convex functions the following fact serves as
the technical pivot in discussing
aggregation, convolution, and transformation by a network
(see \cite{KMT07jump} for details).

\begin{lemma}[{\cite[Lemma 10]{KMT07jump}}]\label{LMeleaggrM}
The elementary aggregation of a jump M-convex function is jump M-convex, 
provided it does not take the value $-\infty$.
\finbox
\end{lemma}

We can extend this lemma to jump \Mnat-convex functions
by using the relation between jump \Mnat- and M-convexity established in Theorem~\ref{THjmnatfn}.

\begin{lemma} \label{LMeleaggrMnat}
The elementary aggregation of a jump \Mnat-convex function is jump \Mnat-convex, 
provided it does not take the value $-\infty$.
\end{lemma}
\begin{proof}
Let 
$f: \ZZ\sp{n} \to \RR \cup \{ +\infty \}$
be a jump \Mnat-convex function, and
$g: \ZZ\sp{n-1} \to \RR \cup \{ +\infty \}$
be its elementary aggregation given by
\begin{equation*} 
g(y_{1}, \dots , y_{n-2}, y_{n-1}) 
= \inf \{ f(y_{1}, \dots , y_{n-2}, x_{n-1}, x_{n})   \mid x_{n-1} + x_{n} = y_{n-1} \}. 
\end{equation*}
As in \eqref{ftilde}
we define 
$\tilde f: \ZZ\sp{n+1} \to \RR \cup \{ +\infty \}$
and 
$\tilde g: \ZZ\sp{n} \to \RR \cup \{ +\infty \}$
 by
\begin{align}  
\tilde f(x_{0}, x) &= 
   \left\{  \begin{array}{ll}
   f(x)            &   (x_{0} = \pi(x)) ,      \\
   + \infty      &   (\mbox{\rm otherwise}) , \\
                      \end{array}  \right.
 \label{ftildeaggr}
\\
\tilde g(y_{0}, y) & = 
   \left\{  \begin{array}{ll}
   g(y)            &   (y_{0} = \pi(y)) ,      \\
   + \infty      &   (\mbox{\rm otherwise}) , \\
                      \end{array}  \right.
 \label{gtildeaggr}
\end{align}
where
$x = (x_{1}, x_{2}, \ldots ,  x_{n})$
and $y = (y_{1}, y_{2}, \dots , y_{n-1})$.
Let $h$ denote the elementary aggregation of $\tilde f$, that is,
\begin{equation*} 
h(y_{0}, y_{1}, \dots , y_{n-2}, y_{n-1}) 
= \inf \{ \tilde f(y_{0}, y_{1}, \dots , y_{n-2}, x_{n-1}, x_{n})   \mid x_{n-1} + x_{n} = y_{n-1} \}. 
\end{equation*}
It turns out that $h$ coincides with $\tilde g$, which we prove later.
Since $f$ is jump \Mnat-convex,
$\tilde f$ is jump M-convex by Theorem~\ref{THjmnatfn} (``only if'' part).
Then, by Lemma \ref{LMeleaggrM},
$h$ is jump M-convex.
By $h =\tilde g$ (shown below),  $\tilde g$ is jump M-convex.
Finally,  Theorem~\ref{THjmnatfn}  (``if'' part) shows that $g$ is jump \Mnat-convex.

It remains to prove $h = \tilde g$.  We use short-hand notations:
\[
y'=(y_{1}, \dots , y_{n-2}),  
\quad
\eta = y_{n-1},
\quad
\xi_{1} = x_{n-1}, 
\quad
\xi_{2} = x_{n}.
\]
Then we have
$y=(y', \eta)$ and
\begin{align} 
g(y', \eta) 
&= \inf \{ f(y', \xi_{1}, \xi_{2})   \mid \xi_{1} + \xi_{2} = \eta \} ,
\label{gaggr}
\\
h(y_{0}, y',\eta) 
&= \inf \{ \tilde f(y_{0}, y', \xi_{1}, \xi_{2})   \mid \xi_{1} + \xi_{2} = \eta \}. 
\label{haggr}
\end{align}
Obviously, we have
$ \pi(y',\eta) = \pi(y', \xi_{1}, \xi_{2})$
if $\xi_{1} + \xi_{2} = \eta$.
If $y_{0} =  \pi(y',\eta)$, 
we have
$ y_{0} = \pi(y', \xi_{1}, \xi_{2})$ and hence
\begin{align*} 
h(y_{0}, y',\eta) 
&= \inf \{ f(y', \xi_{1}, \xi_{2})   \mid \xi_{1} + \xi_{2} = \eta \}
= g(y', \eta) 
= \tilde g(y_{0}, y', \eta) 
\end{align*}
by  
\eqref{ftildeaggr}, 
\eqref{gtildeaggr}, 
\eqref{gaggr}, and \eqref{haggr}.
If $y_{0} \not=  \pi(y',\eta)$, we have 
$\tilde g(y_{0}, y', \eta) = +\infty$ 
by the definition of $\tilde g$ in \eqref{gtildeaggr},
whereas 
$h(y_{0}, y',\eta) = +\infty$ 
since 
$\tilde f(y_{0}, y', \xi_{1}, \xi_{2}) = +\infty$
for all $(\xi_{1}, \xi_{2})$ in \eqref{haggr}.
\end{proof}

It follows from Lemma \ref{LMeleaggrMnat}
that the aggregation of a jump \Mnat-convex function is jump \Mnat-convex.
This extends \cite[Theorem 11]{KMT07jump} for jump M-convex functions.

\begin{theorem} \label{THjmnatfnopAggr}
The aggregation of a jump \Mnat-convex function is jump \Mnat-convex, 
provided it does not take the value $-\infty$.
\finbox
\end{theorem}

Theorem~\ref{THjmnatfnopAggr} implies that if 
a set $S \subseteq \ZZ\sp{N}$ 
is an s.e.~jump system,
the subset of $\ZZ\sp{m}$ defined by 
\begin{equation} \label{aggrsetdef}
 T = \{ (y_{1}, y_{2}, \dots , y_{m}) \in  \ZZ\sp{m} \mid
  x \in S,  \  x(N_{j}) = y_{j} \ (j=1,2,\ldots,m) \}
\end{equation}
is also an s.e.~jump system.

\subsection{Convolution}
\label{SCconvol}

For two functions $f_{1} : \ZZ\sp{N} \to \RR \cup \{ +\infty \}$ 
and $f_{2} : \ZZ \sp{N} \to \RR \cup \{ +\infty \}$, 
we define their (infimum) {\em convolution} as a function 
$f_{1} \Box f_{2} : \ZZ\sp{N} \to \RR \cup \{ +\infty , -\infty\}$ given by 
\begin{equation} \label{convolfndef}
 (f_{1} \Box f_{2})(x) = \inf \{ f_{1}(x_{1}) + f_{2}(x_{2}) 
  \mid x_{1} + x_{2} = x,\  x_{1}\in \ZZ\sp{N}, \ x_{2} \in \ZZ\sp{N} \}. 
\end{equation}

The following theorem states that the convolution operation preserves jump \Mnat-convexity,
which extends \cite[Theorem 12]{KMT07jump} for jump M-convex functions.

\begin{theorem} \label{THjmnatfnopConvol}
The convolution of jump \Mnat-convex functions is jump \Mnat-convex,
provided it does not take the value $-\infty$.
\end{theorem}
\begin{proof}
Consider the direct sum 
$f:\ZZ \sp{N} \times \ZZ \sp{N} \to \RR \cup \{ +\infty \}$ 
of $f_{1}$ and $f_{2}$, which is defined by 
\[
f(x_{1}, x_{2}) = f_{1}(x_{1}) + f_{2}(x_{2}), 
\]
where $x_{1}, x_{2} \in \ZZ\sp{N}$. 
Then $f$ is jump \Mnat-convex by the assumed jump \Mnat-convexity of $f_{1}$ and $f_{2}$. 
Let $\mathcal{P}$ be the partition consisting of pairs of the corresponding elements. 
Then the aggregation of $f$ coincides with $f_{1} \Box f_{2}$.
Hence, by Theorem~\ref{THjmnatfnopAggr}, $f_{1} \Box f_{2}$ is jump \Mnat-convex. 
\end{proof}

Theorem~\ref{THjmnatfnopConvol} implies that 
the Minkowski sum $S_{1} + S_{2}$ of s.e.~jump systems $S_{1}$ and $S_{2}$ 
is an s.e.~jump system.

The composition operation is closely related to convolution.
Let 
$f_{1}: \ZZ\sp{S_{1}} \to \RR \cup \{ +\infty \}$ 
and $f_{2}: \ZZ\sp{S_{2}} \to \RR \cup \{ +\infty \}$,
where
$S_{1}$ and $S_{2}$ are arbitrary (finite) sets.
Put $N_{0} = S_{1} \cap S_{2}$,
$N_{1} = S_{1} \setminus N_{0}$, and $N_{2} = S_{2} \setminus N_{0}$.
The {\em composition} of 
$f_{1}$ and $f_{2}$ is defined as a function 
$g: \ZZ \sp{N_{1} \cup N_{2}}  \to \RR \cup \{ +\infty , -\infty \}$ 
given by
\begin{equation} \label{composfndef}
 g(y_{1}, y_{2}) =
 \inf\{ f_{1}(y_{1} , z_{1}) + 
        f_{2}(y_{2} , z_{2}) \mid
        z_{1} = z_{2} \in \ZZ\sp{N_{0}}
      \} 
\end{equation}
where
$y_{1} \in \ZZ\sp{N_{1}}$ and $y_{2} \in \ZZ\sp{N_{2}}$.

\begin{theorem} \label{THjmnatfnopCompos}
The composition of two jump \Mnat-convex functions is jump \Mnat-convex,
provided it does not take the value $-\infty$
and its effective domain is nonempty.
\end{theorem}
\begin{proof}
Consider functions $h_{1}$ and $h_{2}$ defined by
\begin{align*}
h_{1}(y_{1}, z_{1}, {\bm 0}) 
&= f_{1}(y_{1} , z_{1})  
\qquad
 (y_{1} \in \ZZ\sp{N_{1}}, z_{1} \in \ZZ\sp{N_0}, {\bm 0} \in \ZZ\sp{N_{2}}), 
\\
h_{2}({\bm 0} , z_{2} , y_{2}) 
&= f_{2}(y_{2} , -z_{2}) 
\qquad
({\bm 0} \in \ZZ\sp{N_{1}}, z_{2} \in \ZZ\sp{N_0}, y_{2} \in \ZZ\sp{N_{2}}).
\end{align*}
These functions are jump \Mnat-convex, and hence, by Theorem~\ref{THjmnatfnopConvol},
their convolution $h_{1} \Box h_{2}$ is jump \Mnat-convex.
The restriction of $h_{1} \Box h_{2}$ to $N_{1} \cup N_{2}$,
which is jump \Mnat-convex by Proposition~\ref{PRjmnatfnopRestProj}(1),
coincides with the composition $g$ in \eqref{composfndef}.
\end{proof}

In \cite{BouC95} the composition operation is defined for jump systems,
and it is shown that the composition of two jump systems is a jump system.
Theorem~\ref{THjmnatfnopCompos} implies that
the composition of two s.e.~jump systems is 
an s.e.~jump system.

\begin{remark} \rm \label{RMprojFromConvol}
The convolution operation contains
the projection operation as a special case.
Let
$ g(y)  =  \inf \{ f(y,z) \mid z \in \ZZ\sp{N \setminus U} \}$
be the projection of $f$ to $U$
as defined in \eqref{projfndef}.
Let $\varphi$ be the indicator function of the (cylinder) set
$\{ (y,z) \in \ZZ\sp{U} \times \ZZ\sp{N \setminus U} \mid y = \veczero_{U} \}$,
where 
$\varphi$ is a separable convex function.
The convolution $f \conv \varphi$ is given as
\begin{align*}
(f \conv \varphi)(y,z) &= 
\inf \{ f(y',z') + \varphi(y'',z'')  \mid (y,z) = (y',z') + (y'',z'') \}
\\ & =
\inf \{ f(y',z') \mid  (y,z) = (y',z') + (\veczero,z'') \}
\\ & =
\inf \{ f(y,z - z'') \mid z'' \in \ZZ\sp{N \setminus U} \} 
\\ & =
\inf \{ f(y, z') \mid z' \in \ZZ\sp{N \setminus U} \} 
\\ & =
g(y).
\end{align*}
Thus, the value of  projection $g(y)$
is equal to that of convolution $(f \conv \varphi)(y,z)$ for any $z$. 
In this sense the projection can be regarded as a special case of the convolution.
\finbox
\end{remark}

\subsection{Splitting}
\label{SCsplit}

Suppose that we are given
a family of  disjoint nonempty sets 
$\{U_{i} \mid i \in N \}$ indexed by $N = \{ 1, 2, \dots , n\}$. 
Let $U = \bigcup_{i \in N} U_{i}$. 
For a function $f: \ZZ \sp{N} \to \RR \cup \{+\infty\}$, 
the {\em splitting} of $f$ to $U$ is defined as a function
$g: \ZZ \sp{U} \to \RR \cup \{+\infty\}$ given by
\begin{equation} \label{splitfndef}
 g(y_{1}, y_{2}, \dots , y_{n}) 
  = f( y_{1}(U_{1}), y_{2}(U_{2}), \dots , y_{n}(U_{n})), 
\end{equation}
where $y_{i} = ( y_{ij} \mid  j \in U_{i} )$ is 
an integer vector of dimension $|U_{i}|$
and $y_{i}(U_{i}) = \sum\{  y_{ij} \mid j \in U_{i} \}$ is 
the component sum of vector $y_{i}$. 
If $|U| = n+1$
(in which case we have $|U_{k}|=2$ for some $k$ and $|U_{j}| =1$ for other $j\not= k$),
this is called an {\em elementary splitting}.
Any (general) splitting can be obtained by repeated applications of elementary splittings.

For jump M-convex functions the following fact is 
used in discussing transformation by a network
(see \cite{KMT07jump} for details).

\begin{lemma}[{\cite[Lemma 6]{KMT07jump}}]\label{LMelesplitM}
The elementary splitting of a jump M-convex function is jump M-convex.
\finbox
\end{lemma}

We can extend this lemma to jump \Mnat-convex functions
by using the relation between jump \Mnat- and M-convexity established in Theorem~\ref{THjmnatfn}.

\begin{lemma} \label{LMelesplitMnat}
The elementary splitting of a jump \Mnat-convex function is jump \Mnat-convex.
\end{lemma}
\begin{proof}
Let 
$f: \ZZ\sp{n} \to \RR \cup \{ +\infty \}$
be a jump \Mnat-convex function, and
$g: \ZZ\sp{n+1} \to \RR \cup \{ +\infty \}$
be its elementary splitting given by
\begin{equation*} 
g(y_{1}, \dots , y_{n-1}, y_{n}, y_{n+1}) 
=  f(y_{1}, \dots , y_{n-1}, y_{n} + y_{n+1})  ,
\end{equation*}
where $y_{j} \in \ZZ$ for $j=1,2, \dots , n+1$.
As in \eqref{ftilde}
we define 
$\tilde f: \ZZ\sp{n+1} \to \RR \cup \{ +\infty \}$
and 
$\tilde g: \ZZ\sp{n+2} \to \RR \cup \{ +\infty \}$
 by
\begin{align}
\tilde f(x_{0}, x) &= 
   \left\{  \begin{array}{ll}
   f(x)            &   (x_{0} = \pi(x)) ,      \\
   + \infty      &   (\mbox{\rm otherwise}) , \\
                      \end{array}  \right.
   \label{ftildesplit}
\\ 
\tilde g(y_{0}, y) &= 
   \left\{  \begin{array}{ll}
   g(y)            &   (y_{0} = \pi(y)) ,      \\
   + \infty      &   (\mbox{\rm otherwise}) , \\
                      \end{array}  \right.
   \label{gtildesplit}
\end{align}
where
$x = (x_{1}, x_{2}, \ldots ,  x_{n})$
and $y = (y_{1}, y_{2}, \dots , y_{n+1})$.
Let $h$ denote the elementary splitting of $\tilde f$, that is,
\begin{equation*} 
h(y_{0},y_{1}, \dots , y_{n-1}, y_{n}, y_{n+1}) 
=  \tilde f(y_{0},y_{1}, \dots , y_{n-1}, y_{n} + y_{n+1})  .
\end{equation*}
It turns out that $h$ coincides with $\tilde g$, which we prove later.
Since $f$ is jump \Mnat-convex,
$\tilde f$ is jump M-convex by Theorem~\ref{THjmnatfn} (``only if'' part).
Then, by Lemma \ref{LMelesplitM},
$h$ is jump M-convex.
By $h =\tilde g$ (shown below),  $\tilde g$ is jump M-convex.
Finally,  Theorem~\ref{THjmnatfn}  (``if'' part) shows that $g$ is jump \Mnat-convex.

It remains to prove $h = \tilde g$.  We use short-hand notations:
\[
y'=(y_{1}, \dots , y_{n-1}),  
\quad
\eta_{1} = y_{n},
\quad
\eta_{2} = y_{n+1}.
\]
Then we have
$y=(y', \eta_{1}, \eta_{2})$ and
\begin{align} 
g(y', \eta_{1}, \eta_{2}) 
&= f(y', \eta_{1} + \eta_{2}) ,
\label{gsplit}
\\
h(y_{0}, y',\eta_{1}, \eta_{2}) 
&= \tilde f(y_{0}, y', \eta_{1} + \eta_{2}). 
\label{hsplit}
\end{align}
Obviously, we have
$ \pi(y',\eta_{1}, \eta_{2}) = \pi(y', \eta_{1} + \eta_{2})$.
If $y_{0} =  \pi(y',\eta_{1}, \eta_{2})$, 
we have
$ y_{0} = \pi(y',\eta_{1} + \eta_{2})$ and hence
\begin{align*} 
h(y_{0}, y',\eta_{1}, \eta_{2}) 
&=  f(y', \eta_{1} + \eta_{2}) 
= g(y', \eta_{1} , \eta_{2}) 
= \tilde g(y_{0},y', \eta_{1} , \eta_{2}) 
\end{align*}
by  \eqref{ftildesplit}, \eqref{gtildesplit}, 
\eqref{gsplit}, and \eqref{hsplit}.
If $y_{0} \not=  \pi(y',\eta_{1}, \eta_{2})$, we have 
$\tilde g(y_{0},y', \eta_{1} , \eta_{2}) = +\infty$ 
by the definition of $\tilde g$ in \eqref{gtildesplit},
whereas 
$h(y_{0}, y',\eta_{1}, \eta_{2}) = +\infty$ 
since 
$y_{0} \not=  \pi(y',\eta_{1} + \eta_{2})$
and hence
$\tilde f(y_{0}, y', \eta_{1} + \eta_{2}) = +\infty$
by \eqref{ftildesplit}.
\end{proof}

It follows from Lemma \ref{LMelesplitMnat}
that the splitting of a jump \Mnat-convex function is jump \Mnat-convex.
This extends \cite[Theorem 7]{KMT07jump} for jump M-convex functions.

\begin{theorem}\label{THjmnatfnopSplit}
The splitting of a jump \Mnat-convex function is jump \Mnat-convex.
\finbox
\end{theorem}

Theorem~\ref{THjmnatfnopSplit} implies that if 
a set $S \subseteq \ZZ\sp{N}$ 
is an s.e.~jump system,
the subset of $\ZZ\sp{U}$ defined by 
\begin{equation} \label{splitsetdef}
 T = \{ (y_{1}, y_{2}, \dots , y_{n}) \in  \ZZ\sp{U} \mid
  y_{i} \in \ZZ\sp{U_{i}}, \  x_{i} = y_{i}(U_{i}) \ \ (i \in N) 
   \}
\end{equation}
is also an s.e.~jump system.

\subsection{Transformation by networks}
\label{SCnettrans}

In this section, we consider the transformation of a jump \Mnat-convex function through a network. 
Let $G=(V, A; S, T)$ be a directed graph with vertex set $V$, arc set $A$, 
entrance set $S$, and exit set $T$, where $S$ and $T$ are disjoint subsets of $V$. 
For each $a\in A$, the cost of integer-flow in $a$ is represented by a function 
$\varphi_a : \ZZ \to \RR \cup \{ +\infty \}$, which is assumed to be convex
in the sense of \eqref{univarconvdef}. 

Given a function $f : \ZZ\sp{S} \to \RR \cup \{ +\infty \}$ associated 
with the entrance set $S$ of the network, 
we define a function $g : \ZZ\sp{T} \to \RR \cup \{ + \infty, -\infty \}$ 
on the exit set $T$ by 
\begin{align}
g(y) = \inf_{\xi , x} & \big\{ f(x) + \sum_{a\in A} \varphi_a (\xi (a)) 
\mid 
\partial \xi = (x, -y, \bm{0}), 
\notag \\ &
 \xi \in \ZZ\sp{A}, 
(x, -y, \bm{0}) \in \ZZ\sp{S} \times \ZZ\sp{T} \times \ZZ\sp{V\setminus (S\cup T)} \big\} 
\quad (y \in \ZZ\sp{T}), 
\label{nettransfndef}
\end{align}
where $\partial \xi \in \ZZ\sp{V}$ is the vector of ``net supplies'' given by 
\[
\partial \xi (v) = \sum_{a : \ a \text{ leaves } v} \xi (a) 
- \sum_{a : \ a \text{ enters } v} \xi (a) \quad (v \in V). 
\]
If such $(\xi , x)$ does not exist, we define $g(y) = + \infty$. 
We may think of $g(y)$ as the minimum cost to meet a demand specification $y$ at the exit, 
where the cost consists of two parts, the cost $f(x)$ of supply or production of $x$ 
at the entrance 
and the cost $\sum_{a\in A} \varphi_a (\xi (a))$ of transportation through arcs; 
the sum of these is to be minimized over varying supply $x$ and flow $\xi$ subject to 
the flow conservation constraint $\partial \xi = (x, -y, \bm{0})$. 
We regard $g$ as a result of {\em transformation} (or {\em induction}) of $f$ by the network.

The following theorem shows that 
the transformation of a jump \Mnat-convex function by a network results in 
another jump \Mnat-convex function.
This theorem extends \cite[Theorem 14]{KMT07jump} for jump M-convex functions.

\begin{theorem} \label{THjmnatfnopNettrans}
Assume that $f$ is jump \Mnat-convex and $\varphi_a$ is convex for each $a\in A$. 
Then the function $g$ induced by a network $G=(V, A; S, T)$ is jump \Mnat-convex, 
provided $\dom g \not= \emptyset$ and $g > -\infty$.  
\end{theorem}
\begin{proof}
This fact can be proved based on 
Theorem~\ref{THjmnatfnopAggr} for aggregation, 
Theorem~\ref{THjmnatfnopSplit} for splitting,
and other basic operations.
See the proof of \cite[Theorem 14]{KMT07jump} for the detail.
\end{proof}


\section*{Acknowledgement}

The author thanks Jim Geelen for discussion back in 1996.
He is also thankful to  Yusuke Kobayashi, Kenjiro Takazawa, and Akihisa Tamura
for discussion and comments.
This work was supported by 
CREST, JST, Grant Number JPMJCR14D2, Japan, and
JSPS KAKENHI Grant Number JP20K11697.



\end{document}